\documentclass[12pt]{article}
\usepackage{amsmath,amssymb,amsthm,bbm,natbib,color,hyperref}
\usepackage[margin=1.1in]{geometry}
\usepackage{tikz}
\usetikzlibrary{decorations.pathreplacing}
\usepackage{caption}
\usepackage{graphicx}
\usepackage{subcaption}
\usepackage{afterpage}
\usepackage[hang,flushmargin]{footmisc}
\usepackage{enumerate}
\usepackage{mathbbol}
\usepackage{mathabx}
\newtheorem{theorem}{Theorem}
\newtheorem{lemma}{Lemma}

\newtheorem{assumption}{Assumption}
\theoremstyle{definition}

\newcommand{\R}{\mathbb{R}}
\newcommand{\eps}{\epsilon}
\newcommand{\EE}[1]{\mathbb{E}[{#1}]}

\newcommand{\EEst}[2]{\mathbb{E}[{#1}\  | \ {#2}]}
\newcommand{\PP}[1]{\mathbb{P}\{{#1}\}}
\newcommand{\PPst}[2]{\mathbb{P}\{{#1}\  | \ {#2}\}}
\newcommand{\bigEE}[1]{\mathbb{E}\left[{#1}\right]}
\newcommand{\bigEEst}[2]{\mathbb{E}\left[{#1}\  \middle| \ {#2}\right]}
\newcommand{\bigPP}[1]{\mathbb{P}\left\{{#1}\right\}}

\newcommand{\One}[1]{{\mathbbm{1}}\left\{{#1}\right\}}
\newcommand{\iidsim}{\stackrel{\textnormal{iid}}{\sim}}
\newcommand{\norm}[1]{\|{#1}\|}
\newcommand{\bignorm}[1]{\Big\|{#1}\Big\|}
\newcommand{\bX}{\bar{X}}

\newcommand{\reg}{\mathsf{R}}
\newcommand{\regh}{\widehat{\mathsf{R}}_n}
\newcommand{\Lcal}{\mathcal{L}}
\newcommand{\Lcalh}{\widehat{\mathcal{L}}_n}
\newcommand{\Lcalt}{\widetilde{\mathcal{L}}_n^\rho}
\newcommand{\Lt}{\widetilde{\mathcal{L}}^\rho}

\newcommand{\zLcal}{\mathcal{L}^{0/1}}
\newcommand{\zLcalh}{\widehat{\mathcal{L}}^{0/1}_n}
\newcommand{\zLcalt}{\widetilde{\mathcal{L}}^{0/1}_n}

\newcommand{\Fcal}{\mathcal{F}}
\newcommand{\Yt}{\widetilde{Y}}
\newcommand{\fh}{\widehat{f}}
\newcommand{\ft}{\widetilde{f}}
\newcommand{\wh}{\widehat{w}}
\newcommand{\wt}{\widetilde{w}}
\newcommand{\wts}{\widetilde{w}_\ast}
\DeclareMathOperator{\argmin}{argmin}
\title{Binary classification with corrupted labels}
\author{Yonghoon Lee and Rina Foygel Barber}
\date{\today}

\begin{document}
\maketitle

\begin{abstract}
In a binary classification problem where the goal is to fit an accurate predictor, the presence of corrupted labels in the training data set
may create an additional challenge. However, in settings where likelihood maximization is poorly behaved---for example,
if positive and negative labels are perfectly separable---then a small fraction of corrupted labels can improve performance
by ensuring robustness. In this work, we establish that in such settings, corruption acts as a form of regularization, and 
we compute precise upper bounds on estimation error in the presence of corruptions. Our results
suggest that the presence of corrupted data
points is beneficial only up to a small fraction of the total sample, scaling with the square root of the sample size.
\end{abstract}

\section{Introduction}\label{sec:intro}
Consider a classification problem, where our goal is to predict a binary label $Y\in\{\pm 1\}$ using information
captured by a feature vector $X\in\R^d$.
Based on $n$ training data points $(X_1,Y_1),\dots(X_n,Y_n)$,
the objective is to fit a classifier $\widehat{f}:\R^d\rightarrow\{\pm 1\}$ to this data, mapping a new test feature vector $X$ to a predicted
label $+1$ or $-1$. 

In many settings, inherent noise in the measurement process can introduce corruption into the observed
labels $Y_i$. For example, consider a medical application where features $X_i$ for patient $i$ determine
their likelihood of having a particular disease, and $Y_i\in\{\pm 1\}$ indicates presence or absence of the disease. Imperfect diagnostic
tests might mean that the observed label may differ from the true label $Y_i$. Writing $\Yt_i\in\{\pm 1\}$ to denote the observed
label, we might have $\PPst{\Yt_i=-1}{Y_i=+1}>0$ (if the diagnostic test has a nonzero rate of false negatives) and similarly $\PPst{\Yt_i=+1}{Y_i=-1}>0$
(indicating false positives).

\subsection{Setting and notation}
We begin by introducing some basic notation and definitions that we will use throughout.
Consider the following model for the triples $(X,Y,\Yt)$, where as before, $X\in\R^d$ denotes the feature vector, $Y\in\{\pm 1\}$ is the true
label (which we do not observe), and $\Yt\in\{\pm 1\}$ is the observed label (which may be corrupted, i.e., may differ from the true label):
\begin{align*}
X&\sim P_X \textnormal{\quad (a distribution on $\R^d$)},\\
Y |X &= \begin{cases} +1, &\text{ with prob.~}\eta(X),\\ -1, &\text{ with prob.~}1-\eta(X),\end{cases}\\
\Yt|X,Y &= \begin{cases} -Y, &\text{ with prob.~}\rho,\\ Y, &\text{ with prob.~}1-\rho.\end{cases}\end{align*}
Here $\eta(x)$ denotes the probability of a positive (true) label, 
\[\eta(x) = \PPst{Y=+1}{X=x},\]
while $\rho$ denotes the probability that the observed label is corrupted, assumed to be identical across all data points (the ``homogeneous noise'' setting).

In the classification problem, our goal is to define a classification rule that, given a feature
vector $x\in\R^d$, outputs a predicted label $+1$ or $-1$. The misclassification rate is minimized by predicting
$+1$ or $-1$ depending on whether $\eta(x)$ is above or below 0.5, respectively.
In a real data setting where $\eta(x)$ is unknown, the classification problem is 
typically addressed by fitting some function $f(x)\in\R$
and then predicting the label $\textnormal{sign}(f(x))$. We can interpret $f(x)$ as containing information
about both our prediction for the label (via the sign) and our confidence in this prediction (via the magnitude---values
$f(x)\approx 0$ indicate uncertainty).

Given a possible choice of the function $f$, the misclassification rate on the training data set $\{(X_i,Y_i):i=1,\dots,n\}$ is therefore given by
the empirical 0-1 loss,
\[\zLcalh(f) = \frac{1}{n}\sum_{i=1}^n \One{f(X_i)\cdot Y_i \leq 0},\]
while
\[\zLcalt(f) = \frac{1}{n}\sum_{i=1}^n \One{f(X_i)\cdot \Yt_i \leq 0}\]
measures misclassification on the {\em corrupted} training data set $\{(X_i,\Yt_i):i=1,\dots,n\}$.
Our goal is to ensure a low ``true'' misclassification rate, i.e., for predicting the label $Y$ for a new point with features $X$, that is,
\[\zLcal(f) = \PP{f(X)\cdot Y\leq 0},\]
where $(X,Y)$ is a new data point drawn from the same distribution as the original training data---that is, $X\sim P_X$, and $Y|X$ is
a label in $\{\pm1\}$ with probabilities determined by $\eta(X)$.

Since the zero/one loss is challenging to optimize, it is standard to use a {\em surrogate loss function}
$\ell: \R\rightarrow \R_+$,
typically chosen to be continuous, convex, and monotone nonincreasing.
For example, a logistic surrogate loss is given by
\[\ell(t) = \log(1 + e^{-t}),\]
while the hinge loss is given by
\[\ell(t) = \max\{0, 1 - t\}.\]
Given a sample of $n$ data points, $(X_1,Y_1),\dots,(X_n,Y_n)$, we then define the {\em empirical risk}
\[\Lcalh(f) = \frac{1}{n}\sum_{i=1}^n \ell(f(X_i)\cdot Y_i),\]
which is the average surrogate loss on the data set $\{(X_i,Y_i):i=1,\dots,n\}$,
and the {\em corrupted empirical risk}
\[\Lcalt(f) = \frac{1}{n}\sum_{i=1}^n \ell(f(X_i)\cdot \Yt_i),\]
which is the average surrogate  loss on the {\em corrupted} data set $\{(X_i,\Yt_i):i=1,\dots,n\}$. We will also write
\[\Lcal(f) = \EE{\ell(f(X)\cdot Y)},\]
the ``true'' risk of a function $f$, with expectation taken over a  data point $(X,Y)$ drawn from the same distribution as before, i.e., $X\sim P_X$, and label $Y|X$ drawn with probabilities determined by $\eta(X)$.

\subsection{Summary of questions and results}
The key question of this work is to compare the performance of the empirical risk minimizer,
\[\fh = \argmin_{f\in\Fcal}\Lcalh(f),\]
and its corrupted counterpart,
\[\ft = \argmin_{f\in\Fcal}\Lcalt(f),\]
where the minimization is taken over some predefined class of functions $\Fcal$ (for example, linear functions of $x$).
That is, how does the presence of corrupted labels affect the performance of the empirical risk minimizer?
In particular, we emphasize that the surrogate loss function is unchanged---we do not adjust $\ell$ or attempt to ``correct'' for the presence
of corruption (this is in contrast to much of the existing literature, which we review below).

Our findings can be summarized as follows. 
First, we find that {\bf corruption mimics regularization}---in particular, for a fixed function $f\in\Fcal$, the corrupted empirical risk $\Lcalt(f)$ is a {\em biased} estimate
of the true risk $\Lcal(f)$, but acts as an {\em unbiased} estimate of a penalized version of this risk,
\[\Lcal(f) +\lambda \reg(f)\]
where $\lambda>0$ is a penalty parameter depending  on the corruption level $\rho$, while  the regularization function is given by
\[\reg(f) = \bigEE{\frac{\ell(f(X)) + \ell(-f(X))}{2}},\]
the expected loss of the function $f$ under a completely random label.

While adding a penalty introduces bias into our estimator, it also serves to reduce variance, and for limited sample size $n$,
this reduction in variance may outweigh the bias. Our second finding is therefore that, in some settings, {\bf corruption may lead to reduced risk for finite sample size},
since it is effectively acting as a regularizer and can substantially reduce variance.

\subsection{Prior work}
 The problem of learning a classifier in the presence of corrupted labels has been
studied in many works in the recent literature. Here we give a very brief overview of the settings and types of results
considered. Consider the more general model
\begin{align*}
X&\sim P_X \textnormal{\quad (a distribution on $\R^d$)},\\
Y |X &= \begin{cases} +1, &\text{ with prob.~}\eta(X),\\ -1, &\text{ with prob.~}1-\eta(X),\end{cases}\\
\Yt|X,Y &= \begin{cases} -Y, &\text{ with prob.~}\rho(X,Y),\\ Y, &\text{ with prob.~}1-\rho(X,Y).\end{cases}\end{align*}
Here $\eta(x)$ denotes the probability of a positive (true) label as before, 
while $\rho(x,y)$ denotes the probability that the observed label is corrupted, 
\[\rho(x,y) = \PPst{\Yt\neq Y}{X=x,Y=y},\]
which now may depend on $x$ and/or $y$.

\citet{frenay2014comprehensive} and \citet{6685834} provide overviews of recent works on this problem. They categorize the existing methods to three types: label noise-robust models, data cleaning methods, and label noise-tolerant learning algorithms.

 The {\em homogeneous noise} setting assumes that $\rho(x,y)\equiv \rho$ for all $x,y$---that is, there is a constant probability
for each label to be corrupted. This is the setting we study in the present work. Under this setting, \citet{long2010random} study boosting algorithms and discuss negative consequences of label noise. \citet{van2015learning} consider ERM method and propose a label noise-robust loss function. \citet{manwani2013noise} discuss the noise-tolerance property of risk minimization. \citet{blanco2020mathematical} propose robust algorithms that apply relabeling and clustering to SVM.

The {\em class-dependent noise} setting assumes that $\rho(x,y)=\rho_y$ for all $x,y$---that is, the probability of corrupting a positive
label ($Y=+1$ but $\Yt=-1$) is constant with respect to the feature vector $x$, and similarly for a negative label, but these two probabilities may differ.
For example, in our earlier medical example, the diagnostic test might have different false positive and false negative rates, but these rates
themselves are constant across patients (i.e., independent of features such as age that might be included in the $X$ vector). \citet{Liu_2016}, \citet{scott2013classification}, and \citet{blanchard2016classification} study the consistency of the classifier under corruption, while \citet{reeve2019classification} focus on the minimax optimal learning rate of the corrupted estimator. Some recent works try correction of the loss function or the observed labels; see \citet{JMLR:v18:15-226}, \citet{JMLR:v18:16-315}, \citet{patrini2017making}, and \citet{lin2021learning}. Other recent works focus on studying or developing label noise-robust methods; see \citet{natarajan2013learning}, \citet{patrini2016loss}, \citet{reeve2019fast}, \citet{bootkrajang2012label}, and \citet{bootkrajang2014learning}.

Finally, the {\em general} setting---where $\rho(x,y)$ might vary with $x$---is studied by \citet{cannings2019classification}.
In particular, they examine a setting where the corrupted labels $\Yt_i$ are more ``clean'' than the original labels $Y_i$, in the sense that
the corruption mechanism defined by $\rho(x,y)$ acts to denoise labels near the decision boundary (i.e., $\eta(x)\approx 0.5$) Specifically, suppose that,
for values $x$ with $\eta(x)$ slightly higher than $0.5$, we have 
$\rho(x,+1) < \rho(x,-1)$ (that is, a label $Y_i=-1$ that ``should'' instead be positive, has a greater chance of being flipped to $\Yt_i=+1$),
and similarly if $\eta(x)$ is slightly lower than $0.5$ then $\rho(x,+1)>\rho(x,-1)$. 
In this case, the $\Yt_i$'s carry strictly more information for estimating the decision boundary, as compared to the $Y_i$'s; this setting is therefore
fundamentally different from the one we consider here, where homogeneous noise creates strictly noisier labels. \citet{menon2016learning} consider a similar general setting where they show that any consistent algorithm for noise free setting is also consistent under noisy labels under appropriate assumptions. Recent  discussions on the noise-tolerence and the robustness of the corrupted classification under this setting can be found in \citet{ghosh2015making} and \citet{cheng2020learning}.

\section{Main results}

\subsection{Intuition: corruption acts as regularization}
The key idea for studying the corrupted estimator through the framework of regularization, is to find a regularizer that
matches the expected behavior of the corruption. 
In order to do this, we first find a different representation of the corruption variables:
define
\[R_i \iidsim \textnormal{Bernoulli}(2\rho)\text{ and }Z_i \iidsim \textnormal{Uniform}\{\pm 1\},\]
drawn independently from each other and independently of the clean data. Then let
\[\Yt_i =  (1 -R_i) \cdot Y_i + R_i \cdot Z_i.\]
That is, $R_i$ determines whether the label $Y_i$ will be replaced by a random sign, and $Z_i$ provides this random sign.
Examining this construction we can see that this yields the same distribution of the corrupted labels as the original definition.
We can then write the corrupted loss as
\[\Lcalt(f) = \frac{1}{n}\sum_{i=1}^n \ell(f(X_i)\cdot \Yt_i) 
= \frac{1}{n}\sum_{i=1}^n (1-R_i) \cdot \ell(f(X_i)\cdot Y_i)  + \sum_{i=1}^n R_i \cdot  \ell(f(X_i)\cdot Z_i).\]
Next, we treat $f$ as fixed, and then condition on the clean data and marginalize over the distribution of the $R_i$'s and $Z_i$'s:
\begin{multline*}\bigEEst{\Lcalt(f)}{X_{1:n},Y_{1:n}}\\
= \frac{1}{n}\sum_{i=1}^n \EE{1-R_i} \cdot \ell(f(X_i)\cdot Y_i)  + \frac{1}{n}\sum_{i=1}^n \EE{R_i} \cdot  \EEst{\ell(f(X_i)\cdot Z_i)}{X_i}\\
= (1-2\rho)\cdot\Lcalh(f) + \rho \cdot \frac{1}{n}\sum_{i=1}^n \big(\ell(f(X_i))+\ell(-f(X_i))\big).\end{multline*}
Recall the definition of the regularizer,
\[\reg(f) =\bigEE{\frac{\ell(f(X))+\ell(-f(X))}{2}},\]
the expected loss of $f$ on purely random labels. 
We can also consider an empirical version,
\[\regh(f) = \frac{1}{n}\sum_{i=1}^n\frac{\ell(f(X_i))+\ell(-f(X_i))}{2}.\]
We therefore see that
\[\bigEEst{\Lcalt(f)}{(X_i,Y_i),i=1,\dots,n} = (1-2\rho) \cdot \Big(\Lcalh(f) + \lambda \regh(f)\Big),\]
where $\lambda = \frac{2\rho}{1-2\rho}$.
Finally, for any fixed function $f$, we have
\[\EE{\Lcalh(f) + \lambda \regh(f)} = \Lcal(f) + \lambda \reg(f),\]
by definition. Therefore, we can view the corrupted empirical risk minimizer $\ft$ as a sample estimate
of the minimizer of the penalized loss $ \Lcal(f) + \lambda \reg(f)$. 

To summarize our findings so far, we have seen that $\ft = \argmin_{f\in\Fcal}\Lcalt(f)$
can be described in two ways:
\begin{itemize}
\item Fixing the training data $\{(X_i,Y_i):i=1,\dots,n\}$ and taking an expectation over the corruption mechanism (the $R_i$'s and $Z_i$'s above),
we see that
$\Lcalt(f)$ has (conditional) expected value $\Lcalh(f) + \lambda \regh(f)$, a penalized empirical risk.
\item Taking expectations over both the original data and the random corruption,
$\Lcalt(f)$ has expected value $\Lcal(f) + \lambda \reg(f)$, a penalized true risk.
\end{itemize}

\subsection{Results for the linear setting}
Next, we will examine the implications of this relationship between corruption and regularization, on the goals of minimizing risk.
From this point on, we will restrict our discussion to the setting where $\Fcal$ consists of {\em linear} functions,
\[\Fcal = \{x\mapsto w^\top x: w\in\R^d\},\]
in order to be able to achieve precise results.
Consequently we will shift our notation from functions $f$ to vectors $w$. 
Specifically, for each $w\in\R^d$ we will define the population-level loss and regularized loss,
\[
\Lcal(w) = \EE{\ell(X^\top w\cdot Y)}\textnormal{\quad and \quad}
\Lt(w) = \EE{\ell(X^\top w\cdot Y)} + \frac{2\rho}{1-2\rho}\cdot\reg(w),\]
where
\[\reg(w) = \bigEE{\frac{\ell(X^\top w)+\ell(-X^\top w)}{2}} = \frac{\Lcal(w)+\Lcal(-w)}{2},\]
as well as the empirical loss and empirical corrupted loss,
\[
\Lcalh(w) =\frac{1}{n}\sum_{i=1}^n \ell(X_i^\top w \cdot Y_i)\textnormal{\quad and \quad}
\Lcalt(w) =\frac{1}{n}\sum_{i=1}^n \ell(X_i^\top w \cdot \Yt_i).\]
We will also define population-level minimizers
\begin{equation}\label{eqn:pop_min}w_* = \argmin_{w\in\R^d}\Lcal(w)\textnormal{\quad and \quad}\wts^\rho = \argmin_{w\in\R^d}\Lt(w),\end{equation}
and empirical minimizers
\begin{equation}\label{eqn:emp_min}\wh_n = \argmin_{w\in\R^d}\Lcalh(w)\text{\quad and\quad}\wt_n^\rho=\argmin_{w\in\R^d}\Lcalt(w),\end{equation}
whenever these minimizers exist. (Note that, in some settings,
the loss or its empirical or corrupted counterpart may have no minimizer---for example, logistic loss, where the positive and negative labels can be perfectly separated.) For each of the four minimization problems, if the minimizer exists but is not unique, our results will apply to any minimizer
(e.g., $\wts^\rho$ denotes any element of the set $\argmin_{w\in\R^d}\Lt(w)$, etc).

It is well-known that regularization may help reduce risk, even at the cost of increasing bias due to the influence
of the regularization function.
As discussed earlier, since corruption mimics regularization, in many settings we empirically observe
that corruption reduces the risk---that is, $\Lcal(\wt^\rho_n) < \Lcal(\wh_n)$, even though the corruption introduces bias.
We will next study why this phenomenon occurs, by establishing bounds on the loss $\Lcal(\wt^\rho_n)$ of the corrupted estimator.

\subsubsection{Theoretical results}
We begin by defining our assumptions. First, we require some conditions on the loss function $\ell$:
\begin{assumption}\label{asm:ell}
The loss function $\ell$ is nonnegative, nonincreasing, convex, and $L$-Lipschitz.
Furthermore,  $\ell$
 is strictly decreasing on negative values, with
\[\ell(t) \geq \ell(0)+ \gamma|t|\textnormal{ for all $t\leq 0$}\]
for some $\gamma>0$, and has a subexponential decay for positive values,
\[\ell(t)\leq c_1 e^{-c_2t}\text{ for all $t\geq 0$},\]
for some $c_1,c_2>0$.
\end{assumption}
\noindent The last two conditions ensure
that the loss function enacts a strong penalty if $X^\top w$ predicts the sign of $Y$ incorrectly (i.e., $\ell(t)$ is large for $t< 0$), 
but decays quickly if $X^\top w$ predicts the sign of $Y$ correctly (i.e., $\ell(t)$ is small for $t> 0$). These conditions are satisfied
by many well-known examples, for instance:
\begin{itemize}
\item The logistic loss $\ell_t = \log(1+e^{-t})$ satisfies Assumption~\ref{asm:ell} with $\gamma =\frac{1}{2}$ and $L=c_1=c_2=1$.
\item The hinge loss $\ell_t =  (1-t)_+$ satisfies Assumption~\ref{asm:ell} with $L=\gamma=c_1=c_2=1$.
\end{itemize}
We will also need some weak assumptions on the distribution of the feature vector $X$:
\begin{assumption}\label{asm:X}
For some $a_0,a_1,a_2>0$, it holds that
\[\EE{e^{a_0|X^\top u|^2}} \leq a_1\]
and 
\[\bigEE{e^{-t|X^\top u|}}\leq \frac{a_2}{t}\text{ for all $t>0$}.\]
for all unit vectors $u\in\mathbb{S}^{d-1}$.
\end{assumption}
\noindent For example, this assumption is satisfied
by any multivariate Gaussian distribution with mean $\mu$ and
covariance $\Sigma$, with the parameters $a_0,a_1,a_2$ depending on $\norm{\mu}$
and on the largest and smallest eigenvalues of $\Sigma$, but not on the dimension $d$.

Under these assumptions, our main result establishes a bound on the loss of the corrupted estimator $\wt^\rho_n$.
\begin{theorem}\label{thm:main} Suppose that Assumptions~\ref{asm:ell} and~\ref{asm:X} hold.
Let $n\geq 2$ and fix any $\alpha>0$. 
Suppose $\rho\in(0,\tfrac{1}{2})$ satisfies
\[\rho\geq C \cdot\frac{d\log n}{n}.\]
Then with probability at least $1-n^{-\alpha}$, the set $ \argmin_{w\in\R^d}\Lcalt(w)$ is nonempty,
and for all $\wt_n^\rho \in \argmin_{w\in\R^d}\Lcalt(w)$ it holds that
\[\Lcal(\wt_n^\rho)\leq \inf_{w\in\R^d}\Lcal(w) + C' \left[\rho^{1/2} + \rho^{-1/2}\cdot \sqrt{\frac{d\log n}{n}}\,\right].\]
Here $C,C'$ depend only on $\alpha$ and on the constants in Assumptions~\ref{asm:ell} and~\ref{asm:X}, but not on $n$, $d$, or $\rho$.
\end{theorem}
We can see an immediate tradeoff in the upper bound in Theorem~\ref{thm:main}.
The $\rho^{1/2}$ term acts as an ``approximation error'',
where a large corruption proportion $\rho$ leads to a potentially large gap between the loss of the regularized estimator, $\Lcal(\wts^\rho)$,
and the minimum possible loss without regularization, $\inf_{w\in\R^d}\Lcal(w)$. 
On the other hand, the $\rho^{-1/2}\cdot \sqrt{\frac{d\log n}{n}}$ term is the ``estimation error'', 
which is large when the corruption proportion $\rho$ is small (i.e., insufficient regularization).
 The resulting upper bound on risk is minimized when the corruption level scales as  $\rho\asymp \big(\frac{d\log n}{n}\big)^{1/2}$, leading to an upper bound
 on excess risk scaling as  $\asymp \big(\frac{d\log n}{n}\big)^{1/4}$.  This suggests that even a very small fraction of corrupted
 entries can lead to a reduced risk. In contrast, the uncorrupted minimization problem may not behave well under these weak assumptions---for instance,
 if the labels are perfectly linearly separable (as might be the case if, e.g., $Y|X$ follows a logistic regression with very high signal strength),
 then a minimizer does not even exist (i.e., $\argmin_{w\in\R^d}\Lcalh(w)$ is empty).

 Of course, the result of Theorem~\ref{thm:main} is an upper bound on the loss, and may be loose for certain examples; the value of $\rho$ that minimizes
 the upper bound (i.e.,  $\rho\asymp \big(\frac{d\log n}{n}\big)^{1/2}$) might not be the same as the value of $\rho$ that minimizes the loss itself.
In particular, the result can be viewed as a ``worst case'' bound that holds even when the unregularized loss has no minimizer (such as logistic
regression with perfectly separable labels, as mentioned above); in problems where this is not the case, regularization is not as critical, and a smaller value of $\rho$ (or even $\rho=0$) may perform better.

\subsubsection{Proof of Theorem~\ref{thm:main}}
Our first step is to examine some properties of the regularized population minimizer $\wts^\rho$
and its empirical counterpart, the corrupted estimator $\wt_n^\rho$.
\begin{lemma}\label{lem:main}
Suppose Assumptions~\ref{asm:ell} and~\ref{asm:X} hold. Fix any $\rho\in(0,\frac{1}{2})$. Then
$\argmin_{w\in\R^d}\Lt(w)$ is nonempty, and any $\wts^\rho\in\argmin_{w\in\R^d}\Lt(w)$ 
must satisfy $\norm{\wts^\rho}\leq C_0\rho^{-1/2}$ and
\[\Lcal(\wts^\rho)\leq \inf_{w\in\R^d}\Lcal(w) + C_1 \rho^{1/2}.\]
Moreover,
 for any $\alpha>0$, 
  if $n\geq 2$ and $\rho \geq C\cdot \frac{d\log n}{n}$ then with probability at least $1-n^{-\alpha}$ it holds that
$\argmin_{w\in\R^d}\Lcalt(w)$ is nonempty, that any $\wt_n^\rho\in\argmin_{w\in\R^d}\Lcalt(w)$ 
must satisfy $\norm{\wt_n^\rho}\leq C_0\rho^{-1/2}$, and that
\[\sup_{\norm{w}\leq C_0\rho^{-1/2}}\left|\Lcalt(w)-\Lt(w)\right|\leq C_2 \rho^{-1/2}\sqrt{\frac{d\log n}{n}}.\]
Here $C,C_0,C_1,C_2$ depend on $\alpha$  and on the constants in Assumptions~\ref{asm:ell} and~\ref{asm:X},
but not on $n$, $d$, or $\rho$.
\end{lemma}

Now we prove the theorem. 
By Lemma~\ref{lem:main},
with probability at least $1-n^{-\alpha}$,
for any $\wts^\rho\in\argmin_{w\in\R^d}\Lt(w)$ and all  $\wt_n^\rho\in\argmin_{w\in\R^d}\Lcalt(w)$ it holds that
$\Lcal(\wts^\rho)\leq \inf_{w\in\R^d}\Lcal(w) + C_1 \rho^{1/2}$ and
that
\[\max\left\{\left|\Lcalt(\wts^\rho)-\Lt(\wts^\rho)\right|, \ \left|\Lcalt(\wt_n^\rho)-\Lt(\wt_n^\rho)\right|\right\}\leq C_2 \rho^{-1/2}\sqrt{\frac{d\log n}{n}}.\]
 From now on, we assume that these events all hold. Then we have
\begin{align*}
\Lt(\wt_n^\rho) 
&=\Lt(\wts^\rho) + \left(\Lcalt(\wts^\rho)  -\Lt(\wts^\rho)\right) + \left( \Lcalt(\wt_n^\rho) -\Lcalt(\wts^\rho) \right) + \left(\Lt(\wt_n^\rho) - \Lcalt(\wt_n^\rho) \right) \\
&\leq \Lt(\wts^\rho) + \left( \Lcalt(\wt_n^\rho) -\Lcalt(\wts^\rho) \right)  + 2C_2\rho^{-1/2}\cdot \sqrt{\frac{d\log n}{n}}\\
&\leq \Lt(\wts^\rho)   + 2C_2\rho^{-1/2}\cdot \sqrt{\frac{d\log n}{n}} \textnormal{\quad by optimality of $\wt_n^\rho$}\\
&\leq \inf_{w\in\R^d}\Lcal(w) + C_1 \rho^{1/2} + 2C_2\rho^{-1/2}\cdot \sqrt{\frac{d\log n}{n}}\\
&\leq \inf_{w\in\R^d}\Lcal(w) + \frac{C'}{2} \left[\rho^{1/2} + \rho^{-1/2}\cdot \sqrt{\frac{d\log n}{n}}\,\right],
\end{align*}
where we set $C' = \max\left\{2C_1, 4C_2\right\}$.
Next, by definition of $\Lt$, we have
\begin{multline*}\Lt(\wt_n^\rho)  -  \inf_{w\in\R^d}\Lcal(w) 
= (1-\rho)\cdot \big[\Lcal(\wt_n^\rho)  -  \inf_{w\in\R^d}\Lcal(w) \big] + \rho \cdot \big[\Lcal(-\wt_n^\rho)  -  \inf_{w\in\R^d}\Lcal(w) \big] \\
\geq \frac{1}{2}\big[\Lcal(\wt_n^\rho)  -  \inf_{w\in\R^d}\Lcal(w) \big] \end{multline*}
where the last step holds since $\rho\leq \frac{1}{2}$.
Therefore,
\[\Lcal(\wt_n^\rho)\leq \inf_{w\in\R^d}\Lcal(w) + C' \left[\rho^{1/2} + \rho^{-1/2}\cdot \sqrt{\frac{d\log n}{n}}\,\right] ,\]
which completes the proof of the theorem.

\subsubsection{Another perspective on the regularizer}

The results above suggest that the main source of possible improvements by corruption is the shrinkage induced by the corruption (or, at the population
level, by the regularizer $\reg(w)$). In particular, the results of Lemma~\ref{lem:main} show that, in the linear setting,
the corruption (or the regularizer) lead to an upper bound on $\norm{w}$. We will now examine this connection more closely.

The following lemma verifies that, up to constants, $\reg(w)$ is equivalent to $\norm{w}$.
In a sense, then, we can view regularization with $\reg(w)$ as effectively placing a penalty on $\norm{w}$.
\begin{lemma}\label{lem:Rw}
Suppose Assumptions~\ref{asm:ell} and~\ref{asm:X} hold. 
Then it holds that
\[ \max\{ c_L \cdot \norm{w} , \ell(0)\}  \leq \reg(w) \leq c_U \cdot \norm{w} + \ell(0)\textnormal{ for all $w\in\R^d$},\]
where $c_L,c_U$ depend only on the constants in Assumptions~\ref{asm:ell} and~\ref{asm:X}.
\end{lemma}
\begin{proof} 
In the calculations~\eqref{eqn:E_bound} and~\eqref{eqn:E_lower_bound}
appearing in the proof of Lemma~\ref{lem:main}, we will see that Assumption~\ref{asm:X} implies that
\[\frac{\log 2}{2a_2} \leq \EE{|X^\top u|}\leq \sqrt{\frac{a_1}{a_0}}\]
for all unit vectors $u\in\R^d$.
For any $w\in\R^d$, for the lower bound, we have
\begin{multline*}
 \reg(w) = \bigEE{\frac{\ell(|X^\top w|)+\ell(-|X^\top w|)}{2}} \geq \bigEE{\frac{\ell(-|X^\top w|)}{2}}\geq \bigEE{\frac{\ell(-|X^\top w|) - \ell(0)}{2}} \\\geq \frac{\gamma}{2}\cdot \EE{|X^\top w|} \geq \frac{\gamma\log 2}{4a_2}\cdot\norm{w},\end{multline*}
and furthermore 
\[\reg(w) = \bigEE{\frac{\ell(|X^\top w|)+\ell(-|X^\top w|)}{2}} \geq \ell(0)\]
by convexity of $\ell$. For the upper bound, 
we have
\begin{multline*}
 \reg(w)  = \bigEE{\frac{\ell(|X^\top w|)+\ell(-|X^\top w|)}{2}} \\
 =\ell(0) +  \bigEE{\frac{\ell(-|X^\top w|) - \ell(0)}{2}} +  \bigEE{\frac{\ell(|X^\top w|) - \ell(0)}{2}}\\
  \leq \ell(0) +  \bigEE{\frac{\ell(-|X^\top w|) - \ell(0)}{2}} 
  \leq \ell(0) + \frac{L}{2} \cdot \EE{|X^\top w|} \leq \ell(0) +  \frac{L}{2}   \sqrt{\frac{a_1}{a_0}} \cdot \norm{w}.\end{multline*}

\end{proof}\section{Simulations}\label{sec:sim}

Now we empirically investigate the effect of corruption through a simulation.\footnote{Code to reproduce
this simulation is available at \url{https://www.stat.uchicago.edu/~rina/code/corrupted_labels_sim.R}.} We generate the data $\left\{(X_i,Y_i)\right\}_{1 \leq i \leq n}$ in the following way:
choosing dimension $d=50$, we draw
\begin{align*}
X_i &\sim \mathcal{N}(0,\mathbf{I}_d)\\
Y_i\mid X_i&= \begin{cases} +1,&\textnormal{ with probability }\frac{\exp\{3X_{i1} + 0.5(X_{i2})^3\}}{1+\exp\{3X_{i1} + 0.5(X_{i2})^3\}},\\
 -1,&\textnormal{ with probability }\frac{1}{1+\exp\{3X_{i1} + 0.5(X_{i2})^3\}},\end{cases}
\end{align*}
independently for each $i=1,\dots,n$. 
The corrupted labels $\{\Yt_i\}_{1 \leq i \leq n}$ are generated as
\[\Yt_i\mid X_i,Y_i = \begin{cases} -Y_i, &\text{ with prob.~}\rho,\\ Y_i, &\text{ with prob.~}1-\rho,\end{cases}\]
independently for each $i=1,\dots,n$.
We run the experiment at a small and large sample size, $n=400$ and $n=2000$, 
and at a range of values of the corruption probability, $\rho\in\{0,0.01,0.02,\dots,0.2\}$. For each
sample size $n$ and corruption level $\rho$, we run 100 independent trials of the experiment,
we choose the logistic loss function $\ell(t) = \log(1 + e^{-t})$,
and compute the corrupted empirical minimizer $\wt_n^\rho$ defined in~\eqref{eqn:emp_min} and the penalized population-level minimizer
$\wt_*^\rho$ as in~\eqref{eqn:pop_min} (which reduces to the uncorrupted empirical minimizer $\wh_n$ and the unpenalized population-level minimizer $w_*$,
respectively, in the case $\rho=0$). Note that the data generating distribution does not follow the logistic regression model
(due to the cubic term), and so the logistic loss simply acts as a surrogate for the 0-1 loss (i.e., it
does not correspond to a likelihood for some well-specified model).
\begin{figure}[t]
\begin{center}
\includegraphics[width=\textwidth]{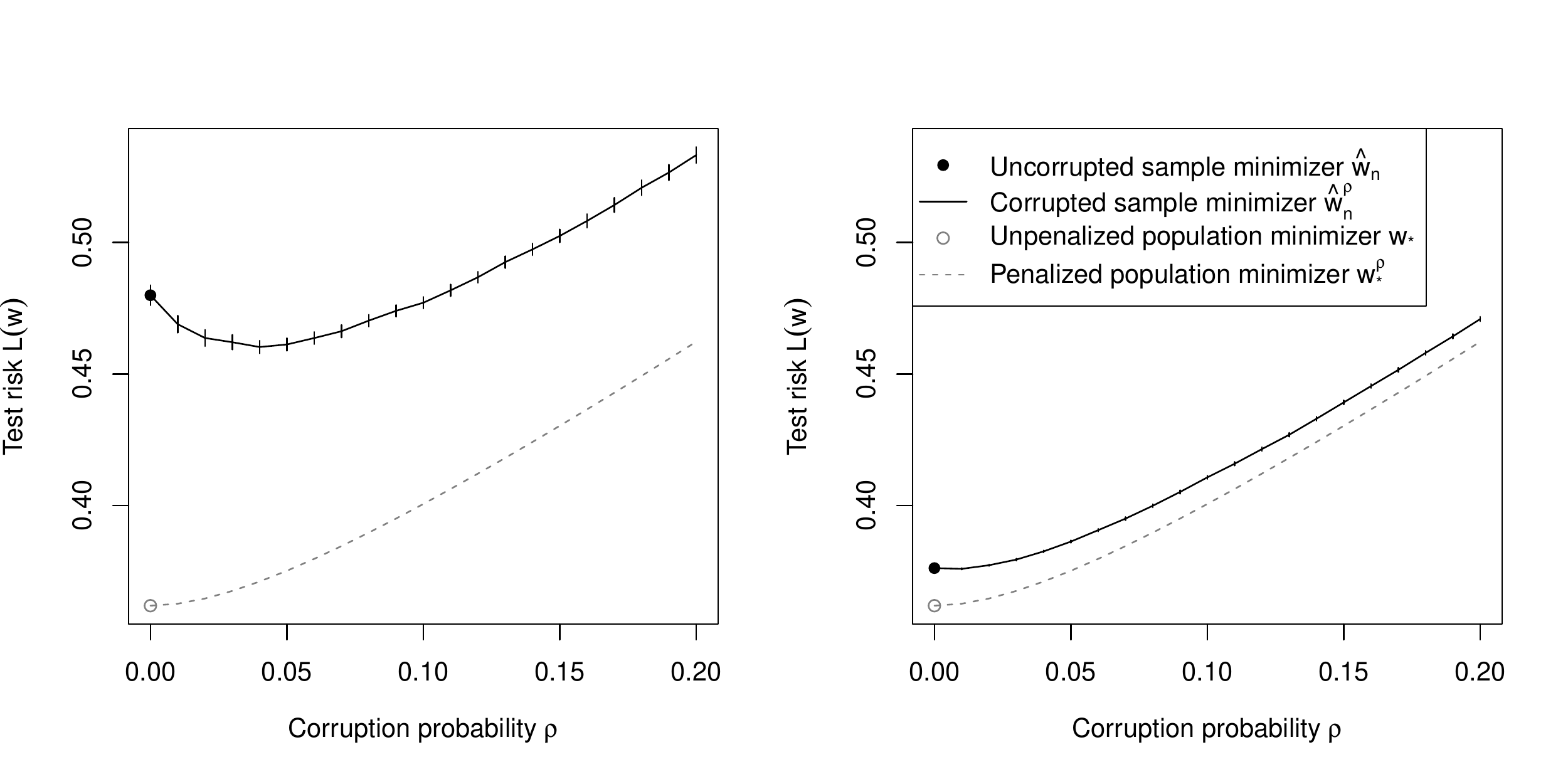}
\caption{Risks of the original classifier $\wh_n$, the corrupted classifier $\wt_n^\rho$, the optimal classifier $w_*$, and the population-level corrupted classifier $\wts^\rho$ on the test set, with sample size $n=400$ (left) and $n=2000$ (right). For the sample estimators $\wh_n$ and $\wt_n^\rho$, the figure displays the mean over 100 independent
trials, with standard error bars.
See Section~\ref{sec:sim} for further details.}
\label{fig:plot1}
\end{center}
\end{figure}

Figure~\ref{fig:plot1} shows the performance of the corrupted estimator $\wt_n^\rho$ and its population-level version $\wt_*^\rho$, across
the range of corruption values $\rho\in\{0,0.01,0.02,\dots,0.2\}$, at each sample size $n\in\{400,2000\}$; the result at $\rho=0$ is highlighted
in each case, as it corresponds to the uncorrupted estimator $\wh_n$ and to the corresponding population-level minimizer $w_*$.
Overall, the plots illustrate how corruption acts as regularization---for the smaller sample size $n=400$, we see
that a small amount of corruption substantially reduces the test risk
 of the empirical minimizer $\wt_n^\rho$, while for the larger sample size $n=2000$ the uncorrupted estimator $\wh_n$
achieves good performance and we no longer see any noticeable improvement from corruption. For the population-level minimizers, on the other hand,
increasing regularization always leads to an increase in risk, as expected.

\section{Discussion}
In this work, we have shown that the corruption of labels has a regularization-type effect on binary classification problems, leading to a possibility of an improvement of the fitted classifier in terms of test risk. Unlike many prior works that apply adjustment or correction to achieve consistency or robustness of the estimator, our result  implies that corruption itself can be beneficial without any adjustment to the estimation process, and thus it could be better in some cases to simply fit the corrupted dataset without any modification on the methods---in particular, this means that we do not need to know or estimate the corruption mechanism, as would be the case for a procedure that corrects for the corruption. For the fitting of linear classifiers using empirical risk minimization under homogeneous noise, Theorem~\ref{thm:main} provides an explanation for the possibility of corruption being beneficial, illustrating the tradeoff between loss approximation and the estimation. 

We can expect a similar tradeoff for more general settings where the noise is not homogeneous, or where different estimation methods are applied; in general, it is intuitive that a small amount of corruption can reduce the chance of overfitting, especially when the inherent noise level is low, and that this benefit may outweigh the low bias that is introduced.
As an example of a broader setting where this type of phenomenon may be useful, we can consider a setting where some data points are known to be ``clean'' while others are potentially corrupted; while we might expect that performance could be improved by removing or down-weighting the latter data points in order to avoid or reduce the effect
of corruption, our findings instead suggest that the presence of the non-``clean'' data might even be beneficial.

The question of corrupted labels, with its possible risks and benefits, is studied only in a very specific setting in our work (i.e., linear prediction rules in low dimensions), and many open questions remain. First,  noting that the corrupted loss can be thought as another surrogate of 0-1 loss, we may ask how corruption affects the prediction performance of the estimator in terms of misclassification rate, i.e., 0-1 risk. 
Second, do similar phenomena occur in the high-dimensional regime, $d\gg n$ or $d\propto n$? In particular, we have seen that homogeneous corruption
mimics an $\ell_2$ penalty in the low-dimensional setting; however, the same is not immediately true in high dimensions, since these results rely on concentration
type arguments that would no longer hold (and, in particular, for $d\gg n$, in general both the uncorrupted data $\{(X_i,Y_i)\}_{1\leq i\leq n}$ and the corrupted data $\{(X_i,\Yt_i)\}_{1\leq i\leq n}$
are perfectly linearly separable, so we cannot expect good performance without some additional constraints or regularization).
Finally, since the key phenomenon underlying our results is the way that homogeneous corruption mimics $\ell_2$ regularization (and therefore, 
corruption induces shrinkage in the resulting estimator), this does not explain any potential benefits from corruption 
if we instead use methods such as a $k$-nearest-neighbor estimator, or other methods where there is no notion of shrinkage; is corruption beneficial
more broadly, by reducing the chance of overfitting in a more general sense? We leave these questions for future work.

\subsection*{Acknowledgements}
R.F.B. was partially supported by the National Science Foundation via grants
DMS-1654076 and DMS-2023109, and by the Office of Naval Research via grant  N00014-20-1-2337.

\bibliography{bib}

\appendix

\section{Additional proofs}
\subsection{Proof of Lemma~\ref{lem:main}}

We first verify that $\Lt$ is $\beta$-Lipschitz, where $\beta = L\sqrt{\frac{a_1}{a_0}}$. 
For any $w\neq w'\in\R^d$ we have
\begin{align*}
\left|\Lt(w) - \Lt(w')\right| &=\left|\EE{\ell(X^\top w\cdot \Yt) - \ell(X^\top w'\cdot \Yt)}\right|\\
&\leq \bigEE{\left|\ell(X^\top w\cdot \Yt) - \ell(X^\top w'\cdot \Yt)\right|} \\
&\leq  \bigEE{L \cdot\left|X^\top w\cdot \Yt-X^\top w'\cdot \Yt\right|}\textnormal{\quad since $\ell$ is $L$-Lipschitz by Assumption~\ref{asm:ell}}\\
&= L\bigEE{\left|X^\top (w-w')\right|}\textnormal{\quad since $\Yt\in\{\pm 1\}$}\\
&=  L \norm{w-w'} \cdot \EE{|X^\top u|}\textnormal{\quad where $u = \frac{w-w'}{\norm{w-w'}}$}\\
&\leq\beta\cdot \norm{w-w'},
\end{align*}
where the last inequality follows from Assumption~\ref{asm:X} via the calculation
\begin{equation}\label{eqn:E_bound}
a_1 \geq \bigEE{e^{a_0|X^\top v|^2}} \geq a_0\cdot\EE{|X^\top v|^2} \geq a_0\cdot\EE{|X^\top v|}^2.
\end{equation}
We therefore have that $\Lt$ is $\beta$-Lipschitz. Note that the above argument also holds for $\rho = 0$, implying that $\Lcal$ is also $\beta$-Lipschitz.

Now fix $t = C_0\rho^{-1/2}$ for any $C_0>\sqrt{\frac{8c_1a_2^2}{c_2\gamma \log 2}}$. We will show that, for any $u\in\mathbb{S}^{d-1}$,
\[\Lt(t\cdot u) > \Lt(0.5t \cdot u).\]
First we calculate
\[ \bigEE{|X^\top u|\cdot \One{X^\top u\cdot \Yt<0}} \geq \rho \cdot \EE{|X^\top u|}\geq \rho\cdot\frac{\log 2}{2a_2}\]
where the first inequality holds by definition of the distribution of the corrupted label $\Yt$
(since $\PPst{\Yt=+1}{X} \in [\rho,1-\rho]$ holds almost surely), 
while for the second inequality, by Jensen's inequality together with Assumption~\ref{asm:X},
\[e^{-2a_2\EE{|X^\top u|}}\leq \EE{e^{-2a_2|X^\top u|}} \leq \frac{a_2}{2a_2} = \frac{1}{2},\]
so 
\begin{equation}\label{eqn:E_lower_bound}
\EE{|X^\top u|} \geq \frac{\log 2}{2a_2}.
\end{equation}
We also know that
\[\ell(-t\cdot |X^\top u|)  - \ell(-0.5t\cdot |X^\top u|)\geq \gamma \cdot 0.5t \cdot |X^\top u|,\]
by Assumption~\ref{asm:ell}, and so
\begin{multline*}\bigEE{\big(\ell(t\cdot X^\top u\cdot\Yt) - \ell(0.5t\cdot X^\top u\cdot\Yt)\big) \cdot \One{X^\top u\cdot \Yt< 0}}\\ \geq \bigEE{ \gamma \cdot 0.5t \cdot |X^\top u|\cdot \One{X^\top u\cdot \Yt<0}}\geq\gamma \cdot 0.5t \cdot \rho \cdot \frac{\log 2}{2a_2}  .\end{multline*}
We therefore have
\begin{align*}
&\Lt(t\cdot u) - \Lt(0.5t\cdot u)\\
&= \bigEE{\ell(t\cdot X^\top u\cdot\Yt) - \ell(0.5t\cdot X^\top u\cdot\Yt)}\\
&= \bigEE{\big(\ell(t\cdot X^\top u\cdot\Yt) - \ell(0.5t\cdot X^\top u\cdot\Yt)\big) \cdot \One{X^\top u\cdot \Yt< 0}}\\
&\hspace{1in} +  \bigEE{\big(\ell(t\cdot X^\top u\cdot\Yt) - \ell(0.5t\cdot X^\top u\cdot\Yt)\big) \cdot \One{X^\top u\cdot \Yt\geq 0}}\\
&\geq \gamma \cdot 0.5t \cdot \rho \cdot \frac{\log 2}{2a_2}+ \bigEE{\big(\ell(t\cdot |X^\top u|) - \ell(0.5t\cdot |X^\top u|)\big) \cdot \One{X^\top u\cdot \Yt\geq 0}}\\
&\geq \gamma \cdot 0.5t \cdot \rho \cdot \frac{\log 2}{2a_2}- \bigEE{\ell(0.5t\cdot |X^\top u|)}\\
&\geq \gamma \cdot 0.5t \cdot \rho \cdot \frac{\log 2}{2a_2}-c_1 \bigEE{e^{-c_2\cdot 0.5t\cdot |X^\top u|}}\textnormal{\ by Assumption~\ref{asm:ell}}\\
&\geq \gamma \cdot 0.5t \cdot \rho \cdot \frac{\log 2}{2a_2} -  \frac{c_1a_2}{c_2\cdot 0.5t}\textnormal{\ by Assumption~\ref{asm:X}}\\
&>0\textnormal{\ by definition of $t$}.
\end{align*}
In particular, this implies that $\Lt(tu)>\inf_{w\in\R^d}\Lt(w)$ for all $u\in\mathbb{S}^{d-1}$.
Since $w\mapsto \Lt(w)$ is continuous as shown above, this implies that $\Lt(w)$ attains its infimum, and any $\wts^\rho\in\argmin_{w\in\R^d}\Lt(w)$ 
must satisfy $\norm{\wts^\rho}\leq t$.

Next we bound $\Lcal(\wts^\rho)$ for any $\wts^\rho\in\argmin_{w\in\R^d}\Lt(w)$. First note that the corrupted risk can be written as
\begin{equation}\label{eqn:rewrite_pen_loss}\Lt(w) =  (1-2\rho)\cdot \Lcal(w) + 2\rho\cdot \reg(w) = (1-\rho)\Lcal(w) + \rho\Lcal(-w).\end{equation}
Applying~\eqref{eqn:rewrite_pen_loss} with $w=\wts^\rho$ we obtain
\[\Lt(\wts^\rho)  = (1-\rho)\Lcal(\wts^\rho) + \rho\Lcal(-\wts^\rho),\]
and similarly applying~\eqref{eqn:rewrite_pen_loss} with $w=-\wts^\rho$ we obtain
\[\Lt(-\wts^\rho) = (1-\rho)\Lcal(-\wts^\rho) + \rho\Lcal(\wts^\rho).\]
Since $\Lt(\wts^\rho)\leq\Lt(-\wts^\rho)$ by optimality of $\wts^\rho$, and $\rho< \frac{1}{2}$ by assumption, this proves that
$\Lcal(\wts^\rho) \leq \Lcal(-\wts^\rho)$ and therefore, 
\[\Lcal(\wts^\rho)\leq \Lt(\wts^\rho).\]
Next, fix any $w\in\R^d$. First consider the case that $\norm{w}\leq  c \rho^{-1/2}$, where $c = \sqrt{\frac{c_1 a_2}{2\beta c_2}}$. Then
\begin{align*}
\Lt(\wts^\rho) - \Lcal(w)
&\leq \Lt(w) - \Lcal(w)\textnormal{\quad by optimality of $\wts^\rho$}\\
&= \rho\left(\Lcal(-w) - \Lcal(w)\right) \textnormal{\quad by~\eqref{eqn:rewrite_pen_loss}}\\
&\leq 2\rho\beta\cdot  c \rho^{-1/2}\\
&= 2 \beta c \rho^{1/2},
\end{align*}
where the last inequality holds since $\Lcal$ is $\beta$-Lipschitz.

Next consider the case that $\norm{w}> c \rho^{-1/2}$.
Let $u=w/\norm{w}$ and $t =  c \rho^{-1/2}$.
Then by the reasoning above, we have
\[\Lt(\wts^\rho) - \Lcal(tu) \leq  2 \beta  c \rho^{1/2}.\]
Next, let $Z_u = X^\top u \cdot Y$, then we have
\begin{align*}
& \Lcal(tu)  - \Lcal(w)
= \EE{\ell(t\cdot Z_u) - \ell(\norm{w}\cdot Z_u)}\\
&= \EE{\left(\ell(t\cdot Z_u) - \ell(\norm{w}\cdot Z_u)\right)\cdot \One{Z_u > 0}} + \EE{\left(\ell(t\cdot Z_u) - \ell(\norm{w}\cdot Z_u)\right)\cdot \One{Z_u < 0}}\\
&\leq \EE{\left(\ell(t\cdot Z_u) - \ell(\norm{w}\cdot Z_u)\right)\cdot \One{Z_u > 0}}\textnormal{\quad since $\norm{w}>t$ and $\ell$ is nonincreasing}\\
&\leq \EE{\ell(t\cdot Z_u) \cdot \One{Z_u > 0}}\textnormal{\quad since $\ell$ is nonnegative}\\
&\leq c_1 \EE{e^{-c_2 t |X^\top u|}}\textnormal{\quad by Assumption~\ref{asm:ell}}\\
&\leq c_1 \cdot \dfrac{a_2}{c_2 t}\textnormal{\quad by Assumption~\ref{asm:X}}\\
&= \frac{c_1 a_2}{c_2 c} \cdot \rho^{1/2}.
\end{align*}
Therefore, for this second case, we have shown that
\[\Lt(\wts^\rho) - \Lcal(w) \leq \left(2\beta c + \frac{c_1 a_2}{c_2 c}\right) \cdot  \rho^{1/2} =  \sqrt{\frac{8\beta c_1 a_2}{c_2}} \cdot \rho^{1/2}.\]
Combining the two cases, we have shown that
\[\Lcal(\wts^\rho)\leq \Lt(\wts^\rho) \leq  \Lcal(w)+  \sqrt{\frac{8\beta c_1 a_2}{c_2}} \cdot \rho^{1/2} \]
for all $w\in\R^d$, which proves the desired inequality with
\[C_1 = \sqrt{\frac{8\beta c_1 a_2}{c_2}}.\]

Now we turn to the corrupted estimator $\wt^\rho_n$.
First we will need a lemma to establish some concentration results.

\begin{lemma}\label{lem:conc}
Suppose Assumptions~\ref{asm:ell} and~\ref{asm:X} hold. Fix any $\alpha>0$, $\rho\in(0,\frac{1}{2})$,  $t>0$, and $r>0$. Then with probability at least $1-n^{-\alpha}$, 
 it holds that
\begin{equation}\label{eqn:conc1}
\inf_{u\in\mathbb{S}^{d-1}} \left\{\frac{1}{n}\sum_{i=1}^n \max\left\{0 , - X_i^\top u \cdot \tilde{Y}_i\right\}\right\} \geq r_1\rho - r_2 \cdot \frac{d \log n}{n}
\end{equation}
and
\begin{equation}\label{eqn:conc2}
\sup_{u\in\mathbb{S}^{d-1}}\left\{\frac{1}{n}\sum_{i=1}^n e^{- t |X_i^\top u|}\right\} \leq \frac{r_3}{t} + r_4\sqrt{\frac{d\log n}{n}}
\end{equation}
and
\begin{equation}\label{eqn:conc3}
\sup_{\norm{w}\leq r} \left|\Lcalt(w) - \Lt(w)\right|\leq r_5 \cdot r \cdot \sqrt{\frac{d\log n}{n}},
\end{equation}
where $r_1,r_2,r_3,r_4,r_5>0$ depend only on $\alpha$ and on the constants in Assumptions~\ref{asm:ell} and~\ref{asm:X},
and not on $n$, $d$, $r$, or $t$.
\end{lemma}

We are now ready to prove the remainder of Lemma~\ref{lem:main}. First we bound $\norm{\wt^\rho_n}$.
Define $C = \frac{2r_2}{r_1}$ and fix $t = C_0\rho^{-1/2}$ for any $C_0>\max\left\{2\sqrt{\frac{4c_1\left(2c_2^{-1}r_3\right)}{\gamma r_1}},\frac{8c_1\left(C^{-1/2} r_4\right)}{\gamma r_1}\right\}$, which therefore
satisfies
\[C_0>\sqrt{\frac{4c_1\left(2c_2^{-1}r_3+C_0 C^{-1/2} r_4\right)}{\gamma r_1}}.\] 
We will show that, for any $u\in\mathbb{S}^{d-1}$,
\[\Lcalt(t\cdot u) > \Lcalt(0.5t \cdot u).\]
Then assuming $\rho\geq C\cdot\frac{d\log n}{n}$,
the bound~\eqref{eqn:conc1} in Lemma~\ref{lem:conc}
implies that 
\[\frac{1}{n}\sum_{i=1}^n |X_i^\top u|\cdot\One{X_i^\top u\cdot \Yt_i <0}  = \frac{1}{n}\sum_{i=1}^n \max\left\{0 , - X_i^\top u \cdot \tilde{Y}_i\right\}\geq \frac{r_1}{2}\cdot \rho,\]
for all $u\in\mathbb{S}^{d-1}$. Furthermore, since $t= C_0\rho^{-1/2}$,
the bound~\eqref{eqn:conc2} in Lemma~\ref{lem:conc} (applied with $0.5c_2t$ in place of $t$) together with our assumption $\rho\geq C\cdot\frac{d\log n}{n}$
implies that 
\[\frac{1}{n}\sum_{i=1}^n e^{- c_2 \cdot 0.5 t |X_i^\top u|} \leq \frac{2c_2^{-1}r_3+C_0 C^{-1/2} r_4}{t}\]
for all $u\in\mathbb{S}^{d-1}$. Following identical arguments as in the population case, we have
\[\Lcalt(t\cdot u) - \Lcalt(0.5t\cdot u)\\
\geq \gamma \cdot 0.5t \cdot \rho \cdot r_1/2 -  c_1\cdot \frac{2c_2^{-1}r_3+C_0 C^{-1/2} r_4}{t}
>0\]
for all $u\in\mathbb{S}^{d-1}$, where the last step holds by definition of $t$ and of $C_0$. Since $\Lcalt$ is continuous (because we have assumed the loss $\ell$ is continuous),
as for the population case this again proves that
$\Lcalt(w)$ must attain its infimum, and that any $w\in\argmin_{w\in\R^d}\Lcalt(w)$ 
must satisfy $\norm{w}\leq t$.

Finally,  the bound $\sup_{\norm{w}\leq C_0\rho^{-1/2}}\left|\Lcalt(w)-\Lt(w)\right|\leq C_2 \rho^{-1/2}\sqrt{\frac{d\log n}{n}}$
follows immediately from the bound~\eqref{eqn:conc3} in Lemma~\ref{lem:conc}, by setting $C_2 = C_0r_5$.

\subsection{Proof of Lemma~\ref{lem:conc}}

First, we prove~\eqref{eqn:conc1}. The distribution of $(X,\Yt)$ can equivalently be represented as
\[(X,\Yt) = \big(X, (1-R)\cdot Y + R\cdot Z\big),\]
where $R\sim\textnormal{Bernoulli}(2\rho)$ is generated independently from $(X,Y)$,
and $Z\sim\textnormal{Unif}\{\pm 1\}$  is generated independently from $(X,Y,R)$.
Let $(X_i,Y_i,R_i,Z_i)$ generate the $n$ i.i.d.~data points. 
Furthermore, define 
\[\bX = X\cdot\min\left\{1, \frac{ 4\EE{\norm{X}}}{\norm{X}}\right\}.\]
and
\[\bX_i = X_i\cdot\min\left\{1, \frac{ 4\EE{\norm{X}}}{\norm{X_i}}\right\}.\]
Then we can check that, for all $u\in\mathbb{S}^{d-1}$,
\[\frac{1}{n}\sum_{i=1}^n \max\left\{0 , - X_i^\top u \cdot \tilde{Y}_i\right\} \geq\frac{1}{n}\sum_{i=1}^n \max\left\{0 , - \bX_i^\top u \cdot \tilde{Y}_i\right\} \geq \frac{1}{n}\sum_{i=1}^n \max\left\{0 , - \bX_i^\top u \cdot R_i \cdot Z_i\right\}.\]
Define
\[\Delta = \sup_{u\in\mathbb{S}^{d-1}} \left|\frac{1}{n}\sum_{i=1}^n \max\left\{0 , - \bX_i^\top u \cdot R_i \cdot Z_i\right\} - \EE{\max\left\{0,-\bX^\top u\cdot R\cdot Z\right\}}\right|.\]
We can verify that, since $\bX,R,Z$ are independent, by definition of their distributions we have
\[ \bigEE{\max\left\{0,-\bX^\top u\cdot R\cdot Z\right\}} \geq \rho\cdot \bigEE{|\bX^\top u|}.\]
Furthermore, by Jensen's inequality, 
\begin{multline*}\exp\left\{-4a_2\bigEE{|\bX^\top u|}\right\} \leq \bigEE{e^{-4a_2|\bX^\top u|}} \leq  \bigEE{e^{-4a_2|X^\top u|}}+ \PP{\norm{X}>4\EE{\norm{X}}} \\\leq \frac{a_2}{4a_2} + \frac{\EE{\norm{X}}}{4\EE{\norm{X}}}=\frac{1}{2},\end{multline*}
where the last inequality applies Assumption~\ref{asm:X} together with Markov's inequality. Rearranging terms, then,
\[ \bigEE{|\bX^\top u|} \geq \frac{\log 2}{4a_2}.\]
Therefore, combining everything we have shown so far, it holds deterministically that
\[\inf_{u\in\mathbb{S}^{d-1}}\left\{\frac{1}{n}\sum_{i=1}^n \max\left\{0 , - X_i^\top u \cdot \tilde{Y}_i\right\} \right\}\geq \rho \cdot \frac{\log 2}{4a_2} - \Delta.\]
Now we need to bound $\Delta$ with high probability.

By the symmetrization inequality \citep[Theorem 2.1]{koltchinskii2011oracle} we have
\[\EE{\Delta}\leq 2 \bigEE{\sup_{u\in\mathbb{S}^{d-1}} \left|\frac{1}{n}\sum_{i=1}^n \xi_i\cdot \max\left\{0 , - \bX_i^\top u \cdot R_i\cdot Z_i\right\}\right|},\]
where the last expectation is taken
with respect to the i.i.d.~data $(\bX_i,\Yt_i)$ as well as i.i.d.~Rademacher random variables $\xi_i\iidsim\textnormal{Unif}\{\pm 1\}$.
Since $t\mapsto \max\{0,-t\}$ is $1$-Lipschitz, the  contraction inequality  \citep[Theorem 2.2]{koltchinskii2011oracle}
verifies that
\[\EE{\Delta}\leq 4 \bigEE{\sup_{u\in\mathbb{S}^{d-1}} \left|\frac{1}{n}\sum_{i=1}^n \xi_i\cdot \bX_i^\top u \cdot R_i\cdot Z_i\right|}.\]
Furthermore, deterministically we have
\[ \left|\frac{1}{n}\sum_{i=1}^n \xi_i \cdot \bX_i^\top u \cdot R_i\cdot Z_i \right|
= \left|u^\top \left(\frac{1}{n}\sum_{i=1}^n \xi_i \cdot R_i\cdot Z_i \cdot \bX_i\right)\right|
\leq \bignorm{\frac{1}{n}\sum_{i=1}^n \xi_i \cdot R_i\cdot Z_i \cdot \bX_i}, \]
and so combining everything so far, we have shown that
\[\EE{\Delta}\leq 4 \bigEE{\bignorm{\frac{1}{n}\sum_{i=1}^n \xi_i \cdot R_i\cdot Z_i \cdot \bX_i}}.\]
Moreover, we can see that $(\bX_i,\xi_i\cdot Z_i)$ is equal in distribution
to $(\bX_i,\xi_i)$ (since $Z_i\in\{\pm 1\}$ while $\xi_i\sim\textnormal{Unif}\{\pm 1\}$ is drawn independently from the data),
and so
\[\EE{\Delta}\leq 4 \bigEE{\bignorm{\frac{1}{n}\sum_{i=1}^n \xi_i \cdot \bX_i \cdot R_i}}.\]
Finally,
\begin{multline*}\bigEE{\bignorm{\frac{1}{n}\sum_{i=1}^n \xi_i  \cdot \bX_i\cdot R_i}}^2\leq\bigEE{\bignorm{\frac{1}{n}\sum_{i=1}^n \xi_i  \cdot \bX_i \cdot R_i}^2}
=\frac{1}{n^2}\sum_{j=1}^d \bigEE{\left(\sum_{i=1}^n \bX_{ij}R_i\xi_i\right)^2} \\= \frac{1}{n^2}\sum_{j=1}^d \sum_{i=1}^n\EE{\bX_{ij}^2R_i^2} = \frac{1}{n^2}\sum_{i=1}^n 2\rho\EE{\norm{\bX_i}^2} \leq \frac{1}{n}\cdot 16\EE{\norm{X}}^2\cdot 2\rho,
\end{multline*}
since by definition, it holds deterministically that $\norm{\bX_i}\leq 4\EE{\norm{X}}$, while $R_i\sim\textnormal{Bernoulli}(2\rho)$ is independent from $X_i$. Combining everything so far,
\[\EE{\Delta}\leq 4\sqrt{\frac{1}{n}\cdot 16\EE{\norm{X}}^2\cdot 2\rho}.\]
Next, since for all $u\in\mathbb{S}^{d-1}$ we have
\[\EE{\max\left\{0 , - \bX^\top u \cdot R\cdot Z\right\}^2} \leq  2\rho \cdot \big(4\EE{\norm{X}}\big)^2\]
and
\[0\leq \max\left\{0 , - \bX^\top u \cdot R\cdot Z\right\} \leq 4\EE{\norm{X}}\text{ almost surely},\]
applying  \citep[Bousquet bound, Section 2.3]{koltchinskii2011oracle} yields the concentration result
\begin{multline*}\bigPP{\Delta \leq \EE{\Delta} + \sqrt{\frac{2\log(3n^{\alpha})\cdot\left(  2\rho \cdot 16\EE{\norm{X}}^2 +   4\EE{\norm{X}}\cdot 2\EE{\Delta} \right)}{n}} + 4\EE{\norm{X}}\cdot \frac{\log(3n^{\alpha})}{3n}}\\ \geq 1-\frac{1}{3n^\alpha}.\end{multline*}
Furthermore, Assumption~\ref{asm:X} together with Jensen's inequality implies \[e^{a_0\EE{\norm{X}^2}/d}\leq e^{a_0\max_{1 \leq j \leq d} \EE{|X_j|^2}} \leq \max_{1 \leq j \leq d} \EE{e^{a_0|X_j|^2}}\leq a_1\] and so  $\EE{\norm{X}}\leq\EE{\norm{X}^2}^{1/2}\leq \sqrt{\frac{d \log a_1}{a_0}}$.
Combined with our bound on $\EE{\Delta}$, we can verify that  this bound can be relaxed to
\[\bigPP{\Delta \leq  r' \left(\sqrt{\rho \cdot \frac{d\log n}{n}} + \frac{d\log n}{n}\right)} \geq 1-\frac{1}{3n^\alpha}\]
where $r'$ is chosen appropriately as a function of $\alpha$, $a_0$, and $a_1$.
Therefore, we have shown that with probability at least $1-\frac{1}{3n^\alpha}$,
\[\inf_{u\in\mathbb{S}^{d-1}}\left\{\frac{1}{n}\sum_{i=1}^n \max\left\{0 , - X_i^\top u \cdot \tilde{Y}_i\right\} \right\}\geq \rho \cdot \frac{\log 2}{4a_2} -  r' \left(\sqrt{\rho \cdot \frac{d\log n}{n}} + \frac{d\log n}{n}\right),\]
which is sufficient to verify~\eqref{eqn:conc1} with $r_1,r_2$ chosen appropriately,
since it holds that $\sqrt{\rho \cdot \frac{d\log n}{n}}  \leq \frac{r''\rho}{2} + \frac{d\log n}{2r'' n}$ for all $r''>0$.

Next we prove~\eqref{eqn:conc2}. 
Note that, comparing the two terms in the desired upper bound and noting that $1/t$ is only dominant if $t\leq \sqrt{\frac{n}{d\log n}}$, we can see that it suffices to prove the result for $t\leq \sqrt{\frac{n}{d\log n}}$,
since $t\mapsto \sup_{u\in\mathbb{S}^{d-1}}\left\{\frac{1}{n}\sum_{i=1}^n e^{- t |X_i^\top u|}\right\}$ is monotone nonincreasing in $t$.

We have
\[\sup_{u\in\mathbb{S}^{d-1}}\left\{\frac{1}{n}\sum_{i=1}^n e^{- t |X_i^\top u|}\right\}\leq \sup_{u\in\mathbb{S}^{d-1}}\left\{\frac{1}{n}\sum_{i=1}^n e^{- t |\bX_i^\top u|}\right\}
,\]
where, changing the definition of $\bX$ and $\bX_i$, we let
\[\bX = X\cdot\min\left\{1, \frac{t\EE{\norm{X}}}{\norm{X}}\right\}.\]
and analogously
\[\bX_i = X_i\cdot\min\left\{1, \frac{t\EE{\norm{X}}}{\norm{X_i}}\right\}.\] 
Next fix $\eps>0$, and take a covering $u_1,\dots,u_M$ of $\mathbb{S}^{d-1}$ such that 
\[\sup_{u\in\mathbb{S}^{d-1}}\left\{\min_{m=1,\dots,M}\norm{u-u_m}\right\}\leq \eps.\]
By \citet[Chapter 15]{lorentz1996constructive}, for any $\eps>0$ we can construct a set with this property of size $M\leq (3/\eps)^d$.
Then for any $u\in\mathbb{S}^{d-1}$, if we find $m$ such that $\norm{u-u_m}\leq \eps$, we have
\[e^{- t |\bX_i^\top u|}  \leq e^{- t |\bX_i^\top u_m|} +  t\norm{\bX_i}\cdot \eps\leq e^{- t |\bX_i^\top u_m|} +  t^2\EE{\norm{X}}\cdot \eps,\]
since $e^{-t|x|}$ is $t$-Lipschitz over $x\in\R$.
Therefore,
\[\sup_{u\in\mathbb{S}^{d-1}}\left\{\frac{1}{n}\sum_{i=1}^n e^{- t |X_i^\top u|}\right\}
\leq  t^2\EE{\norm{X}}\cdot \eps + \max_{m=1,\dots,M}\left\{\frac{1}{n}\sum_{i=1}^n e^{- t |\bX_i^\top u_m|}\right\}.\]
Next, for each $m$, by Hoeffding's inequality,
\[\bigPP{\frac{1}{n}\sum_{i=1}^n e^{- t |\bX_i^\top u_m|}-\EE{e^{- t |\bX^\top u_m|}}> \sqrt{\frac{\log(3Mn^{\alpha})}{2n}}} \leq \frac{1}{3Mn^{\alpha}}.\]
Furthermore,
\[\EE{e^{- t |\bX^\top u_m|}} 
\leq \EE{e^{- t |X^\top u_m|}} + \PP{\norm{X}>t\EE{\norm{X}}} \leq \frac{a_2+1}{t},\]
by applying Assumption~\ref{asm:X} together with Markov's inequality.
Therefore, combining everything, with probability at least $1-\frac{1}{3n^\alpha}$,
\[\sup_{u\in\mathbb{S}^{d-1}}\left\{\frac{1}{n}\sum_{i=1}^n e^{- t |X_i^\top u|}\right\}
\leq  t^2\EE{\norm{X}}\cdot \eps  + \sqrt{\frac{\log(3 \cdot (3/\eps)^d \cdot n^{\alpha})}{2n}} + \frac{a_2+1}{t}.\]
Since we have assumed that $t\leq n$, taking $\eps = n^{-2.5}$ we obtain
\[\sup_{u\in\mathbb{S}^{d-1}}\left\{\frac{1}{n}\sum_{i=1}^n e^{- t |X_i^\top u|}\right\}
\leq  \frac{\EE{\norm{X}}}{\sqrt{n}}+ \sqrt{\frac{\log(3 \cdot (3n^{2.5})^d \cdot n^{\alpha})}{2n}} + \frac{a_2+1}{t},\]
which clearly satisfies~\eqref{eqn:conc2} with $r_3,r_4$ chosen appropriately, since as shown before, $\EE{\norm{X}}\leq  \sqrt{\frac{d \log a_1}{a_0}}$.

Finally we prove~\eqref{eqn:conc3}. 
We first bound the quantity in the expected value. 
We have
\begin{multline*}
\bigEE{\sup_{\norm{w}\leq r} \left|\Lcalt(w) - \Lt(w)\right|}
=\bigEE{\sup_{\norm{w}\leq r} \left|\frac{1}{n}\sum_{i=1}^n \left( \ell(X_i^\top w \cdot \Yt_i) - \EE{\ell(X_i^\top w \cdot \Yt_i)}\right)\right|}\\
\leq 2\bigEE{\sup_{\norm{w}\leq r} \left|\frac{1}{n}\sum_{i=1}^n \xi_i \ell(X_i^\top w \cdot \Yt_i) \right|},
\end{multline*}
by the symmetrization inequality \citep[Theorem 2.1]{koltchinskii2011oracle}, where the last expectation is taken
with respect to the i.i.d.~data $(\bX_i,\Yt_i)$ as well as i.i.d.~Rademacher random variables $\xi_i\iidsim\textnormal{Unif}\{\pm 1\}$.
Next, the  contraction inequality  \citep[Theorem 2.2]{koltchinskii2011oracle}
verifies that
\[\bigEE{\sup_{\norm{w}\leq r} \left|\frac{1}{n}\sum_{i=1}^n \xi_i \ell(X_i^\top w \cdot \Yt_i) \right|} \leq 
2L\bigEE{\sup_{\norm{w}\leq r} \left|\frac{1}{n}\sum_{i=1}^n \xi_i \cdot X_i^\top w \cdot \Yt_i \right|},\]
since $\ell$ is $L$-Lipschitz by Assumption~\ref{asm:ell}. Furthermore, deterministically we have
\[ \left|\frac{1}{n}\sum_{i=1}^n \xi_i \cdot X_i^\top w \cdot \Yt_i \right|
= \left|w^\top \left(\frac{1}{n}\sum_{i=1}^n \xi_i \cdot \Yt_i \cdot X_i\right)\right|
\leq \norm{w}\cdot \bignorm{\frac{1}{n}\sum_{i=1}^n \xi_i \cdot \Yt_i \cdot X_i}, \]
and so combining everything so far, we have shown that
\[\bigEE{\sup_{\norm{w}\leq r} \left|\Lcalt(w) - \Lt(w)\right|}\leq 4Lr \bigEE{\bignorm{\frac{1}{n}\sum_{i=1}^n \xi_i \cdot \Yt_i \cdot X_i}}.\]
Moreover, we can see that $(X_i,\xi_i\cdot\Yt_i)$ is equal in distribution
to $(X_i,\xi_i)$ (since $\Yt_i\in\{\pm 1\}$ while $\xi_i\sim\textnormal{Unif}\{\pm 1\}$ is drawn independently from $(X_i,\Yt_i)$),
and so
\[\bigEE{\sup_{\norm{w}\leq r} \left|\Lcalt(w) - \Lt(w)\right|}\leq 4Lr \bigEE{\bignorm{\frac{1}{n}\sum_{i=1}^n \xi_i  \cdot X_i}}.\]
Finally,
\begin{multline*}\bigEE{\bignorm{\frac{1}{n}\sum_{i=1}^n \xi_i  \cdot X_i}}^2\leq\bigEE{\bignorm{\frac{1}{n}\sum_{i=1}^n \xi_i  \cdot X_i}^2}
=\frac{1}{n^2}\sum_{j=1}^d \bigEE{\left(\sum_{i=1}^n X_{ij}\xi_i\right)^2} \\= \frac{1}{n^2}\sum_{j=1}^d \sum_{i=1}^n\EE{X_{ij}^2} = \frac{1}{n}\EE{\norm{X}^2} \leq \frac{d}{n} \cdot \frac{\log a_1}{a_0},
\end{multline*}
since $\EE{\norm{X}^2}\leq {\frac{d \log a_1}{a_0}}$ as calculated above.
Therefore,
\[\bigEE{\sup_{\norm{w}\leq r} \left|\Lcalt(w) - \Lt(w)\right|}\leq \frac{4Lr\sqrt{\log a_1}}{\sqrt{a_0}}\cdot\sqrt{\frac{d}{n}}.\]

Next we prove that the quantity $\sup_{\norm{w}\leq r} \left|\Lcalt(w) - \Lt(w)\right|$ concentrates around its expectation.
First, let $(X',\Yt')$ be an i.i.d.~draw from the distribution of $(X,\Yt)$. For $\lambda\geq 0$,
we calculate
\begin{multline*}
\bigEE{\frac{1}{2}e^{\lambda \norm{X\Yt-X'\Yt'}} + \frac{1}{2}e^{-\lambda\norm{X\Yt-X'\Yt'}}}
\leq \bigEE{e^{\lambda^2 \norm{X\Yt-X'\Yt'}^2/2}}\\
\leq \bigEE{e^{\lambda^2 \cdot (\norm{X\Yt}^2 + \norm{X'\Yt'}^2)}}
= \bigEE{e^{\lambda^2 \cdot \norm{X\Yt}^2}}^2
= \bigEE{e^{\lambda^2 \cdot \norm{X}^2}}^2\\
= \bigEE{e^{\lambda^2 \cdot \sum_{j=1}^d |X_j|^2}}^2
\leq  \bigEE{\frac{1}{d}\sum_{j=1}^d e^{d\lambda^2 \cdot  |X_j|^2}}^2,
\end{multline*}
by the AM--GM inequality.
Applying Assumption~\ref{asm:X}, we then obtain
\[\bigEE{\frac{1}{2}e^{\lambda \norm{X\Yt-X'\Yt'}} + \frac{1}{2}e^{-\lambda\norm{X\Yt-X'\Yt'}}}
 \leq a_1^{\frac{2\lambda^2 d}{a_0}}\]
 as long as $\lambda^2\leq a_0/d$.
Following the proof of \citet[Theorem 1]{kontorovich2014concentration}, 
since $\sup_{\norm{w}\leq r} \left|\Lcalt(w) - \Lt(w)\right|$ is a $\frac{Lr}{n}$-Lipschitz function of each data point product $X_i\cdot\Yt_i$,
\begin{multline*}\bigPP{\sup_{\norm{w}\leq r} \left|\Lcalt(w) - \Lt(w)\right| - \bigEE{\sup_{\norm{w}\leq r} \left|\Lcalt(w) - \Lt(w)\right|}> \frac{Lr}{n} \cdot \sqrt{\frac{8nd \log a_1 \cdot \log(3n^\alpha)}{a_0}}}\\
\leq \exp\left\{\frac{2n\lambda^2 d\log a_1}{a_0} - \lambda \cdot \sqrt{\frac{8nd \log a_1 \cdot \log(3n^\alpha)}{a_0}} \right\}.\end{multline*}
Taking 
\[\lambda =  \frac{a_0}{4nd\log a_1}\cdot \sqrt{\frac{8nd \log a_1 \cdot \log(3n^\alpha)}{a_0}}\]
(which clearly satisfies $\lambda \leq \sqrt{\frac{a_0}{d}}$ for sufficiently large $n$), this probability is bounded by $\frac{1}{3n^\alpha}$.
(If instead $n$ is not sufficiently large (i.e.,  $\lambda > \sqrt{\frac{a_0}{d}}$), then the guarantee~\eqref{eqn:conc3} holds trivially.)
Combining everything, and choosing $r_5$ appropriately, we have established~\eqref{eqn:conc3}.

\end{document}